\documentclass[a4paper,11pt]{amsart}
\usepackage[utf8x]{inputenc}
\usepackage{amsthm,amsfonts,amsmath,amssymb,enumerate}
\usepackage{verbatim} 
\usepackage{rotating} 

\usepackage{hyperref}
\usepackage{pdfsync}
\newcommand{\C}{\mathbb{C}}
\newcommand{\R}{\mathbb{R}}
\newcommand{\Q}{\mathbb{Q}}
\newcommand{\g}{\mathfrak{g}}
\newcommand{\h}{\mathfrak{h}}
\newcommand{\n}{\mathfrak{n}}
\newcommand{\f}{\mathfrak{f}}

\newcommand{\z}{\mathfrak{z}}

\newcommand{\Der}{\operatorname{Der}}

\newtheorem{theorem}{Theorem}[section]
\newtheorem{lemma}[theorem]{Lemma}

\newtheorem{proposition}[theorem]{Proposition}

\theoremstyle{definition}

\theoremstyle{remark}
\newtheorem{remark}[theorem]{Remark}

\numberwithin{equation}{section}

\begin{document}

\title{A distinguished example of \\ filiform deformation}
\author{Joan Felipe Herrera-Granada}
\address{Universidad Nacional de Colombia sede Manizales}
\email{jfherrerag@unal.edu.co}
\author{Paulo Tirao}
\address{FaMAF - Universidad Nacional de Córdoba / CIEM - CONICET}
\email{ptirao@famaf.unc.edu.ar}
\author{Sonia Vera}
\address{FaMAF - Universidad Nacional de Córdoba / CIEM - CONICET}
\email{svera@famaf.unc.edu.ar}
\date{February 2018}

\maketitle

\begin{abstract}
We exhibit an example of a filiform (complex) Lie algebra of dimension 13 with
all its ideals of codimension 1 being characteristically nilpotent,
and we construct a non trivial filiform deformation of it.
\end{abstract}

\section{Introduction}

It is already well known that a nilpotent Lie algebra $\n$ with a codimension 1 ideal $\h$ of rank $\ge 1$ 
(with non trivial semisimple derivations) is not rigid \cite{C,GH}.
From a semisimple derivation of $\h$ one constructs a non trivial linear deformation of $\n$ \cite{GH}.

If $\n$ is it self of positive rank, then it admits a codimension 1 ideal of positive rank.
So that the class of nilpotent Lie algebras with all its ideals of codimension 1 being characteristically nilpotent
(without any semisimple derivation) is contained in the class of characteristically nilpotent Lie algebras.
Let us call them \emph{strong characteristically nilpotent}.
It is not clear to us how large or how small is this class.
If there exists a rigid nilpotent Lie algebra, in opposition to Vergne's conjecture, 
it must be of this class.

Some time ago, Dietrich Burde told us about a family of filiform Lie algebras $\{\f_n\}$, 
that they suspect to be of this class \cite{BEdG}. 

In this paper we pick the first algebra of their list, $\f_{13}$,
we prove that it is strong characteristically nilpotent
and we construct a filiform non trivial deformation of it.

This example provides further strong evidence supporting Vergne's conjecture.

\section{The algebra $\f_{13}$}

In this paper we consider the filiform Lie algebra $\f_{13}$ (see \cite[\S 5]{BEdG}) defined over $\C$,
which in the basis $\{e_0,e_1,\dots,e_{12}\}$ is defined by:
\footnote{Notice that in \cite{BEdG} the basis is $\{e_1,e_2,\dots,e_{13}\}$.}

{\small
\begin{align*}
[e_0, e_i] &= e_{i+1}, \qquad \text{for $i=1\dots 11$} \\
& \\
[e_1, e_2] &= e_4, & [e_1, e_3] &= e_5, \\
[e_1, e_4] &= \tfrac{9}{10}e_6-e_8, & [e_1, e_5] &= \tfrac{4}{5}e_7-2e_9, \\
[e_1,e_6]  &= \tfrac{5}{7}e_8-\tfrac{335}{126}e_{10}+\tfrac{22105}{15246}e_{12}, & [e_1, e_7] &= \tfrac{9}{14}e_9-\tfrac{125}{42}e_{11}, \\ 
[e_1, e_8] &= \tfrac{7}{12}e_{10}-\tfrac{4421}{1452}e_{12}, & [e_1, e_9] &= \tfrac{8}{15}e_{11}, \\
[e_1, e_{10}] &= \tfrac{27}{55}e_{12}, \\ 
& \\
[e_2, e_3] &= \tfrac{1}{10}e_6+e_8, & [e_2, e_4] &= \tfrac{1}{10}e_7+e_9, \\
[e_2, e_5] &= \tfrac{3}{35}e_8+\tfrac{83}{126}e_{10}-\tfrac{22105}{15246}e_{12}, & [e_2, e_6] &= \tfrac{1}{14}e_9+\tfrac{20}{63}e_{11}, \\
[e_2, e_7] &= \tfrac{5}{84}e_{10}+\tfrac{697}{10164}e_{12}, & [e_2, e_8] &= \tfrac{1}{20}e_{11}, \\ 
[e_2, e_9] &= \tfrac{7}{165}e_{12}, \\ 
& \\
[e_3, e_4] &= \tfrac{1}{70}e_8+\tfrac{43}{126}e_{10}+\tfrac{22105}{15246}e_{12}, & [e_3, e_5] &= \tfrac{1}{70}e_9+\tfrac{43}{126}e_{11}, \\
[e_3, e_6] &= \tfrac{1}{84}e_{10}+\tfrac{7589}{30492}e_{12}, & [e_3, e_7] &= \tfrac{1}{105}e_{11}, \\
[e_3, e_8] &= \tfrac{1}{132}e_{12}, \\ 
& \\
[e_4, e_5] &= \tfrac{1}{420}e_{10}+\tfrac{313}{3388}e_{12}, & [e_4, e_6] &= \tfrac{1}{420}e_{11}, \\
[e_4, e_7] &= \tfrac{3}{1540}e_{12}, \\ 
& \\
[e_5, e_6] &= \tfrac{1}{2310}e_{12}.
\end{align*}
}

\begin{remark}
 In this paper $\C$ might be replaced by $\R$ or even $\Q$ with no further changes. 
 Notice that the algebras $\{\f_n\}$ in \cite{BEdG} are defined over $\Q$.
 Our main interest is Vergne's conjecture over $\C$.
\end{remark}

A generic codimension 1 ideal of $\f_{13}$ is of the form $\h_b=\langle e_0+b e_1,\dots,e_{12}\rangle$, for some $b\in\C$.
More precisely there is only one codimension 1 ideal left, the ideal $\h=\langle e_1,\dots,e_{12}\rangle$.
In this section we show that all these ideals are characteristically nilpotent. 

Let us start with the generic 12-dimensional ideal $\h_b=\langle e_0+b e_1,\dots,e_{12}\rangle$, for some $b\in\C$. 
Define a new basis $\langle f_0,f_1,\dots,f_{11}\rangle$ for $\h_b$ as follows:
\begin{eqnarray*}
f_0=e_0+b e_1,\quad f_1=e_2,\quad f_2=[f_0,f_1],\quad f_3=[f_0,f_2],\quad \dots, \quad f_{11}=[f_0,f_{10}].
\end{eqnarray*}
Notice that $\h_b$ is filiform for all $b$ and its central descending series is 
\[ \h_b^i=\langle f_{i+1},\dots,f_{11} \rangle. \]

In order to show that $\h_b$ is characteristically nilpotent, 
it is useful to consider the 8-dimensional filiform quotient 
\[ \overline{\h_b}=\h_b/\h_b^7. \]
By abuse of notation, the quotient may be described using the (same) basis $\{f_0,\dots,f_7\}$:
{\small
\begin{align*}
[f_0, f_i] &= f_{i+1}, \quad \text{for}\ 1\leq i\leq 6, \\
[f_1, f_2] &= \tfrac{1}{10}f_5-\tfrac{27}{100}b f_6+\left(1+\tfrac{5143}{7000}b^{2}\right)f_{7}, 
 & [f_1, f_3] &= \tfrac{1}{10}f_6-\tfrac{27}{100}b f_7, & [f_1, f_4] &= \tfrac{3}{35}f_7,\\
[f_2, f_3] &= \tfrac{1}{70}f_7.
\end{align*}
}  

\begin{lemma}\label{lemma:hb}
 The quotient $\overline{\h_b}$ is characteristically nilpotent.
\end{lemma}

\begin{proof}
Let $D$ be a derivation of $\overline{\h_b}$.
Since $D$ preserves the central descending series of $\overline{\h_b}$, 
the matrix of $D$ in the basis $\{f_0,\dots,f_7\}$ has the following form: 
\small{
\begin{eqnarray*}
D=\begin{pmatrix}
 m_{1,1} & m_{1,2} & 0 & \cdots & 0\\
 m_{2,1} & m_{2,2}  &  0  & \cdots & 0\\
 m_{3,1} & m_{3,2} & m_{3,3}  & \cdots & 0\\
 \vdots & \vdots & \vdots  & \ddots & \vdots\\
 m_{8,1} & m_{8,2} & m_{8,3} & \cdots & m_{8,8}
 \end{pmatrix}
\end{eqnarray*}
}
For $0\le i<j\le 7$ and $0\le k\le 7$ let
\begin{eqnarray*}
E_{i,j}^{k} = \left(D[f_{i},f_{j}]-[Df_{i},f_{j}]-[f_{i},Df_{j}]\right)[k]
\end{eqnarray*}
be the coordinate respect to $f_k$ of $D[f_{i},f_{j}]-[Df_{i},f_{j}]-[f_{i},Df_{j}]$. 
Since $D$ is a derivation, $E_{i,j}^k=0$ for all $i,j,k$.
 
A direct evaluation yields $E_{1,2}^3=-m_{1,2}$, and hence $m_{1,2}=0$.
Also one directly obtains that, for $1\le i\le 6$,
\[ E_{0,i}^{i+1} = m_{i+2,i+2}-m_{1,1}-m_{i+1,i+1}, \]
and therefore recursively
\begin{equation}\label{eqn:1}
 m_{i+2,i+2} = i m_{1,1} + m_{2,2}, \qquad \text{for $1\le i\le 6$.} 
\end{equation}
Now $E_{1,4}^{7}=\frac{3}{35}m_{8,8}-\frac{3}{35}m_{2,2}-\frac{3}{35}m_{5,5}$, so that 
$m_{8,8}=m_{2,2}+m_{5,5}$ and using \eqref{eqn:1} 
\begin{equation}\label{eqn:2}
m_{8,8}=3 m_{1,1}+2 m_{2,2}.
\end{equation}
By combining \eqref{eqn:1} with \eqref{eqn:2} we get that $m_{2,2}=3 m_{1,1}$
and recursively that
\begin{eqnarray}\label{eqn:3}
m_{i+1,i+1} = (i+2) m_{1,1}, \qquad \text{for $1\le i \le 7$}.
\end{eqnarray}
Therefore $D$ is lower triangular with all its diagonal entries being integer multiples of $m_{1,1}$.

Finally, equations $E_{0,2}^{4}=0$, $E_{0,3}^{5}=0$ and $E_{0,4}^{6}=0$ combined, yield $m_{7,6}=m_{4,3}$. 
And since
\begin{eqnarray*}
E_{1,2}^6=\tfrac{1}{10}m_{7,6} - \tfrac{27}{100}b m_{7,7} + \tfrac{27}{100}b m_{2,2} + \tfrac{27}{100}b m_{3,3} - \tfrac{1}{10}m_{4,3},
\end{eqnarray*}
from \eqref{eqn:3} it follows that $m_{1,1}=0$ and then $D$ is nilpotent.
Consequently, $\h_b$ is characteristically nilpotent.
\end{proof}

Let us now come to consider the single 12-dimensional ideal $\h=\langle e_1,\dots,e_{12}\rangle$.
The proof that it is in fact characteristically nilpotent is analogous to the proof of Lemma \ref{lemma:hb}.
Notice that $\h$ is not filiform, it is 6-step nilpotent, and its central ascending series is given by:
\[
\begin{gathered}
 \z=\langle e_{11},e_{12} \rangle, \quad  \z^2=\langle e_{9},\dots,e_{12} \rangle, \quad \z^3=\langle e_{7},\dots,e_{12} \rangle, \\
 \z^4=\langle e_{5},\dots,e_{12} \rangle, \quad \z^5=\langle e_{3},\dots,e_{12} \rangle, \quad \z^6=\h.
\end{gathered}
\]

\begin{lemma}\label{lemma:h}
 The ideal $\h$ is characteristically nilpotent.
\end{lemma}

\begin{proof}
Let $D$ be a derivation of $\h$. 
Since it preserves the central ascending series of $\h$,
the matrix of $D$ in the basis $\{e_1,\dots,e_{12}\}$ is of the following
$(2\times 2)$-block lower triangular form
\small{
\begin{eqnarray*}
D=\begin{pmatrix}
 m_{1,1}  & m_{1,2} \\
          & m_{2,2} \\
 & &  m_{3,3} & m_{3,4} \\
 & &          & m_{4,4} \\
 & \text{\LARGE$\ast$} & & & \ddots \\
 & & & & & m_{11,11} & m_{11,12} \\
 & & & & & & m_{12,12}
 \end{pmatrix}
\end{eqnarray*}
}
For $1\le i<j\le 12$ and $1\le k\le 12$, let as in the proof of Lemma \ref{lemma:hb},
\begin{eqnarray*}
E_{i,j}^{k} = \left(D[e_{i},e_{j}]-[De_{i},e_{j}]-[e_{i},De_{j}]\right)[k].
\end{eqnarray*}
Recall that $E_{i,j}^k=0$ for all $i,j,k$.

The equations
\begin{eqnarray*}
E_{1,2}^{3}=0,\quad E_{1,4}^{5}=0,\quad E_{1,6}^{7}=0,\quad E_{1,8}^{9}=0,\quad E_{1,10}^{11}=0,
\end{eqnarray*}
considered in the given order, yield directly
\begin{eqnarray*}
m_{3,4}=0,\quad m_{5,6}=0,\quad m_{7,8}=0,\quad m_{9,10}=0,\quad  m_{11,12}=0.
\end{eqnarray*}
Then since $E_{2,4}^{6}=m_{1,2}-\tfrac{1}{9}m_{3,4}$, $m_{3,4}=0$ and thus $D$ is lower triangular. 
\small{
\begin{eqnarray*}
D=\begin{pmatrix}
 m_{1,1} & 0 & \cdots & 0\\
 m_{2,1} & m_{2,2} & \cdots & 0\\
 \vdots & \vdots & \ddots & \vdots\\
 m_{12,1} & m_{12,2} & \cdots & m_{12,12}
 \end{pmatrix}
\end{eqnarray*}
}

Now from the equations $E_{0,j}^{j+2}=0$, for $1\leq j\leq 9$, we get that
\begin{eqnarray}
m_{4,4} &=& m_{1,1}+m_{2,2} \notag \\
m_{5,5} &=& m_{1,1}+m_{3,3} \notag \\
m_{6,6} &=& 2m_{1,1}+m_{2,2} \label{eqn:m66} \\
m_{7,7} &=& 2m_{1,1}+m_{3,3} \label{eqn:m77} \\
m_{8,8} &=& 3m_{1,1}+m_{2,2} \notag \\
m_{9,9} &=& 3m_{1,1}+m_{3,3} \notag \\
m_{10,10} &=& 4m_{1,1}+m_{2,2} \notag \\
m_{11,11} &=& 4m_{1,1}+m_{3,3} \notag \\
m_{12,12} &=& 5m_{1,1}+m_{2,2} \notag 
\end{eqnarray}
And from the equations $E_{1,2}^{5}=0$ and $E_{1,3}^{6}=0$, we get that 
\begin{eqnarray}
m_{6,6} &=& m_{2,2}+m_{3,3} \label{eqn:m66B} \\
m_{7,7} &=& m_{1,1}+2m_{2,2} \label{eqn:m77B}
\end{eqnarray}

From \eqref{eqn:m66} and \eqref{eqn:m66B}, it follows that $m_{3,3}=2m_{1,1}$, 
and from \eqref{eqn:m77} and \eqref{eqn:m77B}, it follows that $m_{2,2}=\tfrac{3}{2}m_{1,1}$.
Therefore
\[ m_{i,i}=\tfrac{(i+1)}{2}m_{1,1}\quad \text{for}\quad 2\leq i\leq 12,\]
and thus
\small{
\begin{eqnarray*}
D=\begin{pmatrix}
 \frac{2}{2}m_{1,1} & 0 & \cdots & \cdots & 0 \\
 m_{2,1} & \frac{3}{2}m_{1,1} & 0 & \cdots & 0 \\
 m_{3,1} & m_{3,2} & \frac{4}{2}m_{1,1} & \cdots & 0 \\
 \vdots & \vdots & \vdots & \ddots & \vdots \\
 m_{12,1} & m_{12,2} & m_{12,3} & \cdots & \frac{13}{2}m_{1,1}
 \end{pmatrix}
\end{eqnarray*}
}

We show now that $m_{1,1}=0$.
To this end we consider the following sequence of equations and its consequences.

First, from $E_{0,1}^{5}=0$ we get $m_{6,4}=-\tfrac{1}{10}m_{3,1}+\tfrac{9}{10}m_{4,2}$, 
and then from $9E_{1,2}^{7}-E_{0,3}^{7}=0$ we get that
\[ E1=10m_{1,1}+\tfrac{27}{35}m_{4,2}-\tfrac{27}{35}m_{5,3}-\tfrac{2}{35}m_{3,1}=0. \]

Second, from $E_{0,2}^{6}=0$ we get $m_{7,5}=\tfrac{4}{5}m_{5,3}$, 
and then from $8E_{1,3}^{8}-E_{0,4}^{8}=0$ we get that
\[ E2=10m_{1,1}-\tfrac{18}{35}m_{4,2}+\tfrac{18}{35}m_{5,3}+\tfrac{1}{14}m_{3,1}=0. \]

Now from $\tfrac{3}{2}E2+E1=0$ it follows that
\[ E3=25m_{1,1}+\tfrac{1}{20}m_{3,1}=0 \]

From $E_{0,3}^{7}=0$ we get $m_{8,6}=\tfrac{10}{9}m_{1,1}-\tfrac{4}{63}m_{3,1}+\tfrac{5}{7}m_{4,2}$, 
and then from
$\tfrac{3}{25}E_{0,5}^{9}-E_{1,4}^{9}=0$ we get that  
\[ E4=-\tfrac{19}{18}m_{1,1}-\tfrac{1}{21}m_{4,2}+\tfrac{1}{21}m_{5,3}+\tfrac{19}{6300}m_{3,1}=0. \]

Finally, from $\tfrac{35}{18}E2-21 E4$ we get that
\[ E5=\tfrac{749}{18}m_{1,1}+\tfrac{17}{225}m_{3,1}=0 \]
and from $20 E3-\tfrac{225}{17}E5=0$ it follows that $m_{1,1}=0$, 
and therefore $D$ is nilpotent. 
\end{proof}

\begin{lemma}\label{lemma:n}
 Let $\n$ be a nilpotent Lie algebra and $\n^i$ the $i$-th term of its central descending series.
 If the quotient $\n/\n^i$ is characteristically nilpotent, then $\n$ is characteristically nilpotent.
\end{lemma}

\begin{proof}
Any nilpotent Lie algebra $\g$ is generated, as a Lie algebra, by the set $\g-[\g,\g]$. 
So that $\n$ is generated by $S=\n-[\n,\n]$ and $\n/\n^i$ by the projection set $\overline{S}$.

A non trivial semisimple derivation $D$ of $\n$ is non trivial in $S$,
hence the induced derivation on $\n/\n^i$ is semisimple and non trivial in $\overline{S}$.

Therefore if the quotient algebra $\n/\n^i$ has no semisimple derivations, 
then the algebra $\n$ has not semisimple derivations as well.
\end{proof}

Lemmas \ref{lemma:h}, \ref{lemma:hb} and \ref{lemma:n} prove the following result.

\begin{proposition}
 The algebra $\f_{13}$ is strong characteristically nilpotent.
\end{proposition}

\section{A non-trivial deformation of $\f_{13}$}

In this section let $\f=\f_{13}$.
For the construction of a non trivial deformation of $\f$ we follow closely the construction given in \cite[\S 4]{TV}.
In this case $\h=[\f,\f]=\langle e_2,\dots,e_{12}\rangle$ and we choose
$D\in\Der(\h)$ given by
\[ D:e_2\mapsto e_9, \quad e_3\mapsto e_{10}, \quad e_4\mapsto e_{11}, \quad e_5\mapsto e_{12} \]
and $D(e_i)=0$, for $i=6\dots 12$.
So that, for $t\in\C$ the deformed algebra $\f_t$ is given by the bracket $[\ ,\ ]_t$ described by:
\[ [e_0,e_1]_t=[e_0,e_1], \qquad [e_0,h]_t=[e_0,h], \qquad [e_1,h]_t=[e_1,h]+tD(h), \] 
and $[h,h']_t=[h,h']$, for all $h,h'\in\h$.

It turns out that $\f_t\simeq\f$ if and only if $t=0$.
To prove this, let $g$ be an isomorphism between $\f_t$ and $\f$, and let $[g]$
be its matrix with respect to the basis $\{e_0,\dots,e_{12}\}$ of $\f_t$ and $\f$. 
Since $g$ preserves the central descending series, then $[g]$ is of the form:
\begin{equation}\label{eqn:g}
[g] = \begin{pmatrix}
 m_{1,1} & m_{1,2} & 0 & \cdots & 0\\
 m_{2,1} & m_{2,2}  &  0  & \cdots & 0\\
 m_{3,1} & m_{3,2} & m_{3,3}  & \cdots & 0\\
 \vdots & \vdots & \vdots  & \ddots & \vdots\\
 m_{13,1} & m_{13,2} & m_{13,3} & \cdots & m_{13,13}
 \end{pmatrix}
\end{equation}

That $g$ is an isomorphism is equivalent to 
\[ E_{i,j}=g[e_i,e_j]_t-[ge_i,ge_j]=0 \]
for all $0\le i<j\le 12$.
If $E_{i,j}^k$ is the coefficient of $e_k$ in $E_{i,j}$, 
that is \[ E_{i,j} = g[e_i,e_j]_t-[ge_i,ge_j]=\sum_k E_{i,j}^k e_k, \]
then $g$ is an isomorphism if and only if $E_{i,j}^k=0$, 
for all $0\le i<j\le 12$ and $0\le k\le 12$.

We show first that the matrix $[g]$ above is lower triangular
and its diagonal entries are all equal to 1 or equal to $\pm 1$.

\begin{lemma}\label{lemma:mii}
 In the matrix $[g]$ in \eqref{eqn:g}, $m_{1,2}=0$ and $m_{i,i}=m_{1,1}^{i}$, 
 for $2\le i\le 13$.
 Moreover, $m_{2,1}=0$ and $m_{1,1}=1$ or $m_{1,1}=-1$. 
\end{lemma}

\begin{proof}
That $m_{1,2}=0$ follows directly from the identity 
\[ 0=E_{1,11}^{12}=-m_{1,2}m_{12,12}, \]
which is easy to evaluate.
So we set $m_{1,2}=0$.

By a direct computation, we get that 
\[ E_{0,i}^{i+1}=m_{i+2,i+2}-m_{1,1}m_{i+1,i+1}, \]
for $1 \le i\le 11$,
and hence $m_{i,i}=m_{2,2}m_{1,1}^{i-2}$, for $2\le i\le 13$.
So we have in addition that 
\[ 0=E_{1,2}^4=m_{1,1}^3m_{2,2}-m_{1,1}m_{2,2}^2, \] 
from which it follows that
$m_{2,2}=m_{1,1}^2$ and therefore $m_{i,i}=m_{1,1}^i$, for $1\le i\le 13$.

For the second part we consider, in the given order, the following equations:
\[ E_{0,1}^3=0,\quad E_{0,2}^4=0,\quad E_{0,3}^5=0, \]
from which we get that
\begin{eqnarray*}
 m_{4,3} &=& m_{1,1}m_{3,2} \\
 m_{5,4} &=& m_{1,1}^3m_{2,1}+m_{1,1}^2m_{3,2} \\
 m_{6,5} &=& 2 m_{1,1}^4 m_{2,1} +m_{1,1}^3 m_{3,2}
\end{eqnarray*}
Now, from the identity 
\[ 0=E_{1,2}^5 = 2 m_{1,1}^4 m_{2,1}, \] 
we obtain that $m_{2,1} = 0$.

We continue by considering the equations
\[ E_{0,1}^4 = 0,\quad E_{0,2}^5 = 0,\quad E_{0,3}^6=0, \quad E_{0,4}^7=0, \quad E_{0,5}^8=0, \]
from where we get that
\begin{eqnarray*}
 m_{5,3} &=& -m_{1,1}^2m_{3,1}+m_{1,1}m_{4,2} \\
 m_{6,3} &=& -m_{1,1}^3m_{3,1}+m_{1,1}^2m_{4,2} \\
 m_{7,5} &=& -\frac{9}{10}m_{1,1}^4m_{3,1}+m_{1,1}^3m_{4,2} \\
 m_{8,6} &=& - \frac{4}{5} m_{1,1}^5m_{3,1}+m_{1,1}^4m_{4,2} \\
 m_{9,7} &=& -\frac{5}{7}m_{1,1}^6m_{3,1}+m_{1,1}^5m_{4,2}
\end{eqnarray*}
Now, from the identities
\begin{eqnarray*}
 0 &=& E_{1,2}^6 = \frac{1}{10}m_{1,1}(2m_{1,1}^2m_{4,2}-m_{3,2}^2), \\
 0 &=& E_{1,4}^8 = -\frac{1}{35}m_{1,1}^3(-6m_{1,1}^2m_{4,2}+35m_{1,1}^6-35m_{1,1}^4+3m_{3,2}^2),
\end{eqnarray*}
it follows that
\begin{eqnarray*}
 E1 &=& 2m_{1,1}^2m_{4,2}-m_{3,2}^2=0, \\
 E2 &=& -6m_{1,1}^2m_{4,2}+35m_{1,1}^6-35m_{1,1}^4+3m_{3,2}^2=0.
\end{eqnarray*}
Finally, 
\[ 3 E1 + E2 = 35m_{1,1}^4(m_{1,1}-1)(m_{1,1}+1) \]
and therefore $m_{1,1}=1$ or $m_{1,1}=-1$.
\end{proof}

We continue with another technical result that will simplify all the arguments afterwards.

\begin{lemma}\label{lemma:m32}
 The entries $m_{i+1,i}$ for $2\le i\le 12$, of the second diagonal of the matrix $[g]$, are all equal to $m_ {3,2}$ if $m_{1,1}=1$,
 or $m_{i+1,i}=(-1)^i m_{3,2}$ if $m_{1,1}=-1$. 
\end{lemma}

\begin{proof}
On the one hand, if $m_{i,i}=1$ for all $i$,
we then have that for $1\le i\le 10$
 \[ 0=E_{0,i}^{i+2}=m_{i+3,i+2}-m_{i+2,i+1}, \]
from which it follows that $m_{i+1,i}=m_{3,2}$, for $2\le i\le 12$.

On the other hand, if $m_{i,i}=(-1)^i$,
we then have that for $1\le i\le 10$
 \[ 0=E_{0,i}^{i+2}=m_{i+3,i+2}+(-1)^{i+1}m_{i+2,i+1}, \]
from which it follows that $m_{i+1,i}=(-1)^im_{3,2}$, for $2\le i\le 12$.
\end{proof}

These lemmas all together yield two possible forms for the matrix $[g]$, depending on whether $m[1,1]=1$ or $m[1,1]=-1$:
\begin{itemize}
 \item 
If $m_{1,1}=1$, then
{ \[ [g]= \begin{pmatrix}
  1 & 0 & 0 & 0 & 0   \\
  0 &  1 &  0  & 0 & 0  \\
 m_{3,1} & m_{3,2} &  1  & 0 & 0 & \dots \\
 m_{4,1} & m_{4,2} & m_{3,2} &  1 & 0  \\
 m_{5,1} & m_{5,2} & m_{5,3} & m_{3,2} &  1  \\
  & & \vdots & & \ddots & \ddots \\
 \end{pmatrix} 
\]}

\item
If $m_{1,1}=-1$, then
{ \[ [g]= \begin{pmatrix}
  -1 & 0 & 0 & 0 & 0   \\
  0 &  1 &  0  & 0 & 0  \\
 m_{3,1} & m_{3,2} &  -1  & 0 & 0 & \dots \\
 m_{4,1} & m_{4,2} & -m_{3,2} &  1 & 0  \\
 m_{5,1} & m_{5,2} & m_{5,3} & m_{3,2} &  -1  \\
  & & \vdots & & \ddots & \ddots \\
 \end{pmatrix} 
\]}
\end{itemize}

We are ready to prove the main result of this section.

\begin{theorem}
The filiform Lie algebra $\f_t$ is isomorphic to $\f$ if and only if $t=0$.
\end{theorem}

\begin{proof}
The proof follows by repeatedly considering particular sets of identities
$\{E_{i,j}^k=0\}$ and solving some of the entries $\{m_{i,j}\}$
until one gets that $t=0$.
From now on let $\alpha=m_{1,1}$ and recall that $\alpha=1$ or $\alpha=-1$.

\begin{itemize}
 
 \item By considering the equations
\[ E_{0,1}^4=0,\quad E_{0,2}^5=0,\quad E_{0,3}^6=0, \]
we get that
\begin{eqnarray*}
 m_{5,3} &=& -m_{3,1}+\alpha m_{4,2}, \\
 m_{6,4} &=& -\alpha m_{3,1}+m_{4,2}, \\
 m_{7,5} &=& -\tfrac{9}{10} m_{3,1}+\alpha m_{4,2}. \\
 \end{eqnarray*} 
And from  
\[ 0 = E_{1,2}^6= \alpha \big( -\tfrac{1}{10}m_{3,2}^2 + \tfrac{1}{5}m_{4,2} \big), \]
it follows that
\[ m_{4,2}=\tfrac{1}{2}m_{3,2}^2. \]

 \item By considering the equations
\[ E_{0,1}^5=0,\quad E_{0,2}^6=0,\quad E_{0,3}^7=0,\]
we get that
\begin{eqnarray*}
 m_{6,3} &=& \alpha m_{5,2}-m_{4,1} \\
 m_{7,4} &=& \alpha \big(\tfrac{1}{10} m_{3,1}m_{3,2}-\tfrac{11}{10} m_{4,1}\big)+m_{5,2} \\
 m_{8,5} &=& \alpha m_{5,2}-\tfrac{11}{10} m_{4,1}+\tfrac{1}{5} m_{3,1}m_{3,2} \\
 \end{eqnarray*}
And from $0=E_{1,2}^7$, it follows that
\[ m_{5,2}=\alpha \big(m_{4,1}-m_{3,1}m_{3,2}\big)+\tfrac{1}{6}m_{3,2}^3. \]

 \item By considering the equations
\[ E_{0,1}^6=0,\quad E_{0,2}^7=0,\quad E_{0,3}^8=0, \]
we get that
\begin{eqnarray*}
 m_{7,3} &=& \alpha m_{6,2}-\tfrac{9}{10}m_{5,1}-\tfrac{1}{10}m_{4,1}m_{3,2}+\tfrac{1}{20}m_{3,1}m_{3,2}^2 \\
 m_{8,4} &=& \alpha \big( -m_{5,1}-\tfrac{1}{10}m_{4,1}m_{3,2}+\tfrac{1}{10}m_{3,1}m_{3,2}^2\big)+m_{6,2}-\tfrac{1}{10}m_{3,1}^2  \\
 m_{9,5} &=& \alpha \big(m_{6,2}-\tfrac{13}{70}m_{3,1}^2\big)-\tfrac{71}{70}m_{5,1}+m_{3,1}-\tfrac{3}{35}m_{4,1}m_{3,2}+\tfrac{1}{7}m_{3,1}m_{3,2}^2 \\
 \end{eqnarray*}
And from $0=E_ {1,2}^8$, it follows that
it follows that
\[ m_{6,2}=\alpha \big(m_{5,1}-\tfrac{1}{2}m_{3,1}m_{3,2}^2\big)+\tfrac{1}{2}m_{3,1}^2+\tfrac{1}{24}m_{3,2}^4. \]

 \item By considering the equations
\[ E_{0,1}^7=0,\quad E_{0,2}^8=0,\quad E_{0,3}^9=0, \]
we get that
\begin{eqnarray*}
 m_{8,3} &=& \alpha \big(m_{7,2}+\tfrac{1}{10}m_{3,1}m_{4,1}-\tfrac{1}{10}m_{3,1}^2m_{3,2}\big)
               -\tfrac{1}{10}m_{5,1}m_{3,2} \\ 
         &&    +\tfrac{1}{60}m_{3,1}m_{3,2}^3-\tfrac{4}{5}m_{6,1} \\
 m_{9,4} &=& \alpha \big(-m_{4,1}-\tfrac{31}{35}m_{6,1}+m_{3,1}m_{3,2}+\tfrac{1}{140}m_{4,1}m_{3,2}^2+\tfrac{13}{420}m_{3,1}m_{3,2}^3 \\
         &&    -\tfrac{4}{35}m_{5,1}m_{3,2} \big) +m_{7,2}+\tfrac{3}{35}m_{3,1}m_{4,1}-\tfrac{13}{70}m_{3,1}^2m_{3,2} \\
 m_{10,5} &=& \alpha \big(m_{7,2}+\tfrac{9}{140}m_{3,1}m_{4,1}-\tfrac{1}{4}m_{3,1}^2m_{3,2}\big)-m_{4,1}-\tfrac{9}{10}m_{6,1} \\ 
          &&  +2m_{3,1}m_{3,2}+\tfrac{1}{70}m_{4,1}m_{3,2}^2+\tfrac{3}{70}m_{3,1}m_{3,2}^3-\tfrac{4}{35}m_{5,1}m_{3,2} \\
 \end{eqnarray*}
And from $0=E_{1,2}^9$, it follows that
\begin{eqnarray*}
 m_{7,2} &=& \alpha \big(-\tfrac{7}{3}t+\tfrac{9}{10}m_{6,1}+\tfrac{1}{10}m_{5,1}m_{3,2}-\tfrac{1}{20}m_{4,1}m_{3,2}^2
           -\tfrac{3}{20}m_{3,1}m_{3,2}^3\big) \\
          && +\tfrac{1}{2}m_{3,1}^2m_{3,2}+\tfrac{1}{120}m_{3,2}^5
\end{eqnarray*}

 \item Finally, by considering the equations
\[ E_{0,4}^7=0,\ E_{0,4}^8=0,\ E_{0,5}^9=0,\ E_{0,6}^{10}=0,\ E_{0,7}^{11}=0,\ E_{0,4}^{10}=0,\ E_{0,5}^{11}=0 \]
we get that
\begin{eqnarray*}
m_{8,6} &=& -\tfrac{4}{5}\alpha m_{3,1}+\tfrac{1}{2}m_{3,2}^2\\
m_{9,6} &=& \alpha \big(-\tfrac{3}{35}m_{4,1}-\tfrac{5}{7}m_{3,1}m_{3,2}\big)+\tfrac{1}{6}m_{3,2}^3 \\
m_{10,7} &=& \tfrac{1}{6}\alpha m_{3,2}^3-\tfrac{1}{14}m_{4,1}-\tfrac{9}{14}m_{3,1}m_{3,2} \\
m_{11,8} &=& \alpha \big(-\tfrac{5}{84}m_{4,1}-\tfrac{7}{12}m_{3,1}m_{3,2}\big)+\tfrac{1}{6}m_{3,2}^3 \\
m_{12,9} &=& \tfrac{1}{6}\alpha m_{3,2}^3-\tfrac{1}{20}m_{4,1}-\tfrac{8}{15}m_{3,1}m_{3,2} \\
m_{11,6} &=& \alpha \big(-\tfrac{7}{3}t-\tfrac{83}{126}m_{4,1}-\tfrac{1}{420}m_{6,1}+\tfrac{335}{126}m_{3,1}m_{3,2}-\tfrac{1}{84}m_{5,1}m_{3,2}\\
 && -\tfrac{5}{168}m_{4,1}m_{3,2}^2-\tfrac{7}{72}m_{3,1}m_{3,2}^3\big)+\tfrac{1}{120}m_{3,2}^5+\tfrac{17}{84}m_{3,1}^2m_{3,2}+\tfrac{1}{21}m_{3,1}m_{4,1} \\
m_{12,7} &=& \alpha \big(\tfrac{1}{28}m_{3,1}m_{4,1}+\tfrac{1}{6}m_{3,1}^2m_{3,2}+\tfrac{1}{120}m_{3,2}^5\big) -\tfrac{7}{3}t -\tfrac{20}{63}m_{4,1} \\
         && -\tfrac{1}{420}m_{6,1}+\tfrac{125}{42}m_{3,1}m_{3,2} -\tfrac{1}{105}m_{5,1}m_{3,2}-\tfrac{1}{40}m_{4,1}m_{3,2}^2-\tfrac{4}{45}m_{3,1}m_{3,2}^3 \\
\end{eqnarray*}
And from $0=E_{2,3}^{11}=-\tfrac{5}{36}t$, it follows that $t=0$.
\end{itemize}
\end{proof}



\begin{thebibliography}{BEdG}

\bibitem[BEdG]{BEdG} Burde, D., Eick, B.\ and de Graaf, W.,
   \emph{Computing faithful representations for nilpotent Lie algebras}, 
   J.\ Algebra 322 (2009), no.\ 3, 602--612.  

\bibitem[C]{C} Carles, R., 
   \emph{Sur la structure des alg\`{e}bres de Lie rigides},
   Ann.\ Inst.\ Fourier 34 (1984), no.\ 3, 65--82.

\bibitem[GH]{GH} Grunewald, F.\ and O'Halloran, J.,
   \emph{Deformations of Lie algebras}, 
   J.\ Algebra 162 (1993), no.\ 1, 210--224.    
      
\bibitem[GK]{GK} Goze M.\ and Khakimdjanov Y.,
   \emph{Sur les algèbres de Lie nilpotentes admettant un tore de dérivations},
   Manuscripta Mathematica (1994), Volume 84, Issue 1, 115--124.

\bibitem[TV]{TV} Tirao P.\ and Vera S.,
  \emph{There are no rigid filiform Lie algebras of low dimension},
   arXiv:1709.04793.
     
\end{thebibliography}
\end{document}